\newcommand{\dataversione}{March 13, 2014}

\documentclass[11pt]{amsart}
\usepackage{mathrsfs}
\usepackage{hyperref}\hypersetup{colorlinks=true, citecolor=blue}
\usepackage{graphicx}
\numberwithin{equation}{section}

\newtheoremstyle{mytheorem}
{}
{}
{\it}
{\parindent}
{\bf}
{.}
{ }
{\thmnumber{#2.~}\thmname{#1}\thmnote{~\rm#3}}

\newtheoremstyle{myremark}
{}
{}
{\rm}
{\parindent}
{\bf}
{.}
{ }
{\thmnumber{#2.~}\thmname{#1}\thmnote{~\rm#3}}

\newtheoremstyle{myparagraph}
{}
{}
{\rm}
{\parindent}
{\bf}
{}
{ }
{\thmnumber{#2.~}\thmname{#1}\thmnote{#3}}

\theoremstyle{mytheorem}

\newtheorem{theorem}[subsection]{Theorem}
\newtheorem{lemma}[subsection]{Lemma}
\newtheorem{claim}[subsection]{Claim}

\newtheorem{conone}[subsection]{First Conjecture}
\newtheorem{contwo}[subsection]{Second Conjecture}

\theoremstyle{myremark}

\newtheorem*{remark*}{Remark}

\theoremstyle{myparagraph}

\newtheorem*{parag*}{}

\makeatletter
\def\@secnumfont{\sc}
\def\section{\@startsection{section}{1}%
\z@{1.5\linespacing\@plus .2\linespacing}{.7\linespacing}%
{\normalfont\sc\centering}}
\makeatother

\makeatletter
\def\ps@headings{\ps@empty
 \def\@evenhead{%
  \setTrue{runhead}%
  \normalfont\footnotesize
  \rlap{\thepage}\hfil
  \def\thanks{\protect\thanks@warning}%
  \leftmark{}{}\hfil}%
 \def\@oddhead{%
  \setTrue{runhead}%
  \normalfont\footnotesize\hfil
  \def\thanks{\protect\thanks@warning}%
  \rightmark{}{}\hfil \llap{\thepage}}%
\let\@mkboth\markboth}
\makeatother
\makeatletter
\renewenvironment{proof}[1][\proofname]{\par
  \pushQED{\qed}%
  \normalfont \topsep6\p@\@plus6\p@\relax
  \trivlist
  \itemindent\normalparindent
  \item[\hskip\labelsep
    \scshape
    #1\@addpunct{.}]\ignorespaces
}{%
  \popQED\endtrivlist\@endpefalse
}
\providecommand{\proofname}{Proof}
\makeatother
\newcommand{\R}{\mathbb{R}}

\newcommand{\dV}{d_V\kern-1pt}

\makeatletter
\@addtoreset{footnote}{section}
\makeatother

\begin{document}

	%
\pagestyle{empty}
\pagestyle{myheadings}
\markboth%
{\underline{\centerline{\hfill\footnotesize%
\textsc{Andrea Marchese}%
\vphantom{,}\hfill}}}%
{\underline{\centerline{\hfill\footnotesize%
\textsc{Building dimension and Almgren's stratification}%
\vphantom{,}\hfill}}}

	%
\thispagestyle{empty}

~\vskip -1.1 cm

	%
\centerline{\footnotesize version: \dataversione%
\hfill
}

\vspace{1.7 cm}

	%
{\Large\sl\centering
On the building dimension of closed cones
\\
and Almgren's stratification principle.
\\
}

\vspace{.4 cm}

	%
\centerline{\sc Andrea Marchese}

\vspace{.8 cm}

{\rightskip 1 cm
\leftskip 1 cm
\parindent 0 pt
\footnotesize

	%
{\sc Abstract.}
In this paper we disprove a conjecture stated in \cite{W} on the equality of two notions of dimension for closed cones. Moreover, we answer in the negative to the following question, raised in the same paper. 
Given a compact family $\mathcal{C}$ of closed cones and a set $S$ such that every blow-up of $S$ at every point $x\in S$ is contained in some element of $\mathcal{C}$, is it true that the dimension of $S$ is smaller than or equal to the largest dimension of a vector space contained is some element of $\mathcal{C}$?
\par
\medskip\noindent
{\sc Keywords:} Building dimension, Stratification.
\par
\medskip\noindent
{\sc MSC (2010): 49Q05, 49Q20.}

\par
}

%
%
\section{Introduction}
In \cite{W}, White proves a generalization of Almgren's stratification principle (see \cite{Al}). To explain it, let us consider the simplified Euclidean setting.
A \emph{cone} in $\R^n$ is a set $C$ such that $\lambda x\in C$ for every $x\in C$ and every $\lambda\geq0$.
Given a set $S\subset\R^n$ the \emph{difference set} of $S$ is the set
$${\rm{diff}}(S):=\{x-y: x,y\in S\}.$$
The \emph{building dimension} bdim$(C)$ of a cone $C$ is the quantity
$${\rm{bdim}}(C):=\sup\{\dim_{\mathcal{H}}(S): S\subset\R^n, {\rm{diff}}(S)\subset C\}.$$
Here and throughout the paper, $\dim_{\mathcal{H}}$ denotes the Hausdorff dimension of a subset of $\R^n$,
with respect to the Euclidean distance. Note that the building dimension of a cone $C$ is always less than or
equal to $\dim_{\mathcal{H}}(C)$. Moreover, if $C$ is a vector subspace of $\R^n$, then the building dimension
and the Hausdorff dimension of $C$ are equal, because ${\rm{diff}}(C)=C$. In \cite{W} this notion is used to prove a stratification principle. In particular the following statement is proved (see Theorem 2.2 of \cite{W}). 
We will give the definition of blow-up of a set in \S \ref{bu}.

\begin{theorem}[(White)]\label{t1}
Let $\mathcal{C}$ be a compact family of closed cones such that at each point of a set $S$ each blow-up of $S$ is contained in some element of $\mathcal{C}$.
Then the largest possible Hausdorff dimension $d$ of the set $S$ is smaller than or equal to the largest building dimension of the cones in $\mathcal{C}$.
\end{theorem}

A natural question then becomes the estimate of the building dimension. In the same paper the following characterization in terms of linear subspaces is proved (see Theorem 2.1 and Proposition A.2 of \cite{W}).
\begin{theorem}[(White)]\label{t2}
The building dimension of a closed cone $C$ equals the supremum of the dimensions of the vector
subspaces contained in $C$, in the following two cases:
\begin{itemize}
 \item $C$ is convex;
 \item $C$ is a subset of $\R^2$.
\end{itemize}
\end{theorem}
Moreover, the author conjectures the following.

\begin{conone}
 The conclusion of Theorem \ref{t2} holds for any closed cone $C$.
\end{conone}

The validity of this conjecture would improve the conclusion of Theorem \ref{t1}, proving that $d$ is exactly the largest dimension of a 
linear subspace that is a subset of one of the cones in $\mathcal{C}$. In fact, White conjectures also the following.

\begin{contwo}
If the hypothesis of Theorem \ref{t1} is satisfied, then the largest possible Hausdorff dimension $d$ of the set $S$ equals the largest dimension of a 
linear subspace that is a subset of one of the cones in $\mathcal{C}$.
\end{contwo}

In this paper, we disprove both conjectures. In \S \ref{cex}, we exhibit a closed cone $C\subset\R^3$ which does not contain planes, but whose building dimension is strictly larger than 1.
In \S\ref{sc}, we exhibit a set $S\subset \R^3$ with Hausdorff dimension strictly larger than 1 and with the property that there exists a closed cone which does not contain planes, but which
contains every blow-up of $S$ at every point $x\in S$. Note that the counterexample given in \S \ref{sc}, together with Theorem \ref{t1}, gives an indirect proof that the First Conjecture is not valid.
We give also the direct proof contained in \S \ref{cex}, because we find the construction of some interest. To our knowledge is not known whether or not Theorem \ref{t1} gives the best possible
estimate for the value of the dimension $d$.

%
%
\section*{Acknowledgements}
The author warmly thanks M. Focardi and E. Spadaro for having pointed out the problem and for valuable suggestions.
%
%

\section{Construction of the counterexample to the First Conjecture}\label{cex}
\subsection{Strategy of the proof}
We firstly construct a set $E\subset[-1,1]$ which is a closed, self-similar fractal in the sense of Hutchinson (see \cite{Fa}, \S 8.3), with $\dim_{\mathcal{H}}(E)>\frac{1}{3}$. In our case it is easy to compute the Hausdorff dimension of the Cartesian product
$S:=E\times E\times E\subset\R^3$ and we have $\dim_{\mathcal{H}}(S)=3\dim_{\mathcal{H}}(E)>1$. Then we construct a closed set $F\subset [-2,2]$ satisfying diff$(E)\subset F$ and we exhibit a line $\ell\subset\R^2$ which intersects the cone in $\R^2$ over $F\times F$
only at the origin. This is sufficient to prove that the cone $C$ in $\R^3$ over $F\times F\times F$ does not contain any plane $\pi$. In fact, if such a plane $\pi$ exists, then its projection on at least one of the coordinate planes should be the full coordinate
plane, which is not the case because of the existence of the line $\ell$. On the other hand, $C$ necessarily contains diff$(S)$, therefore $C$ disproves the conjecture.\\

\subsection{Construction of $E$}
Let $A\subset[0,1]$ be the self-similar set obtained as a countable intersection of the sets $A_i$ defined as follows:
\begin{itemize}
 \item $A_0=[0,1]$
 \item for every $i=1,2,\ldots$, $A_i$ is the union of $2^i$ closed intervals $I_{i,0},\ldots,I_{i,2^i-1}$, where $I_{i,2j}$ (respectively, $I_{i,2j+1}$) has length $\frac{1}{8-\varepsilon}$ times the length of
$I_{i-1,j}$ ($\varepsilon>0$ is to be fixed later) and it has the same left (respectively, right) extreme as $I_{i-1,j}$.
\end{itemize}

The set $E$ is the union of $A$ and $-A$. By Theorem 8.6 of \cite{Fa} it is easy to compute $$\dim_{\mathcal{H}}(E)=\dim_{\mathcal{H}}(A)=\frac{\ln 2}{\ln (8-\varepsilon)}>\frac{1}{3},$$ 
because $A$ is the invariant set associated to 2 similitudes, each with ratio $\frac{1}{8-\varepsilon}$. Analogously, denoting $S:=E\times E\times E$, we can compute 
$$\dim_{\mathcal{H}}(S)=\dim_{\mathcal{H}}(A\times A\times A)=\frac{\ln 8}{\ln (8-\varepsilon)}>1,$$ because $A\times A\times A$ is the invariant set associated to 8 similitudes, each with ratio $\frac{1}{8-\varepsilon}$.\\

\subsection{Construction of $F$}
For $i=1,2,\ldots$, we denote $F_i:=$diff$(A_i\cup-A_i)$ and $F=\bigcap_i F_i$. It follows immediately that diff$(E)\subset F$. Indeed in this case there holds diff$(E)=F$, but the inclusion is sufficient to our aim. Moreover $F$ is closed, because each $F_i$ is so, being $F_i$ the image of the Cartesian product $(A_i\cup-A_i)^2$ under a continuous map (the sum of the coordinates). Note that the difference set of a Borel set is not necessarily a Borel set (see \cite{ES}).

\subsection{Final contradiction}
The crucial part of the proof is the following claim.
\begin{claim}
There exists a line $\ell$ on $\R^2$ which intersects the cone in $\R^2$ over $F\times F$
only at the origin. 
\end{claim}

To prove it, denote by $\widetilde{F}_1$ the set $(F_1\times F_1)\setminus G_1$, where $G_1$ is the connected component of $(F_1\times F_1)$ containing the origin. For every $i=2,3,\ldots$,
denote by $\widetilde{F}_i$ the set $((F_i\times F_i)\setminus \widetilde{F}_{i-1})\setminus G_i$, where $G_i$ is the connected component of $(F_i\times F_i)$ containing the origin. Note that for $i=2,3,\ldots$ we have $\widetilde{F}_i\subset G_{i-1}$. Clearly the cone over
$(F\times F)$ is contained in the cone $\widetilde{C}$ over the set $\bigcup_i \widetilde{F}_i$, which coincides with the union of the cones $\widetilde{C}_i$ over $\widetilde{F}_i$.
The key observation is the following lemma.

\begin{figure}[htbp]
\center
\input{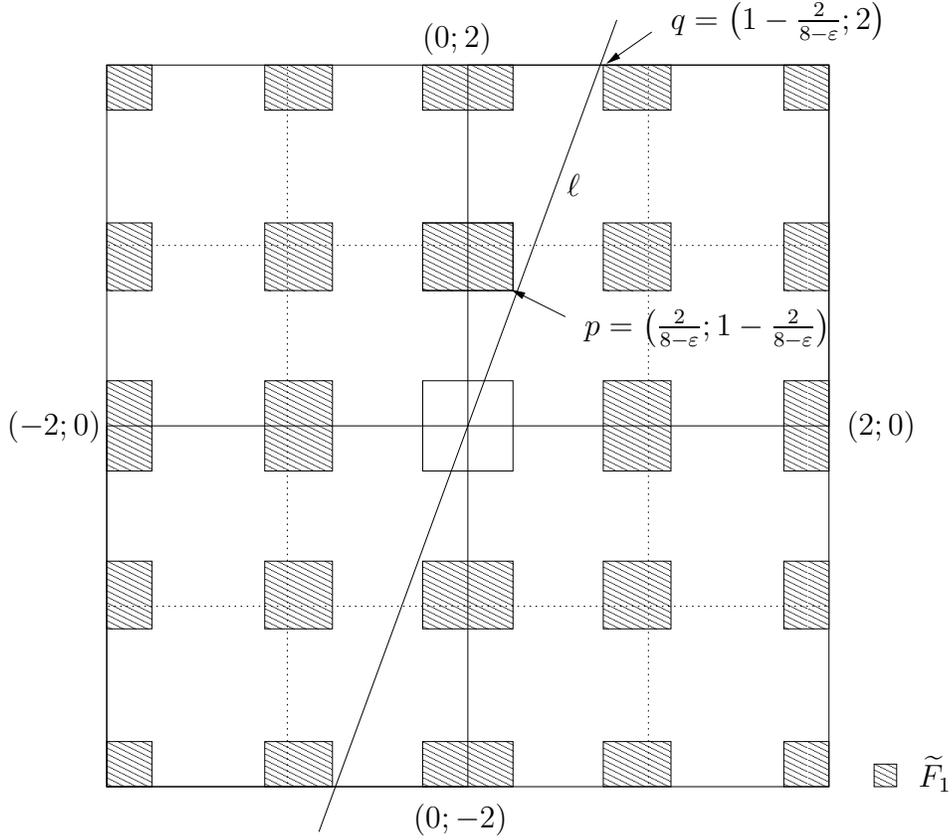}
\caption{A line $\ell$ intersects the cone over $\widetilde{F}_1$ only at the origin. }
\label{f1}
\end{figure}

\begin{lemma}
For every $i=1,2,\ldots$, $\widetilde{F}_{i+1}$ is an homothetic copy of $\widetilde{F}_{i}$. More precisely, $\widetilde{F}_{i+1}=\frac{1}{8-\varepsilon}\widetilde{F}_{i}.$
\end{lemma}
\begin{proof}
It is sufficient to prove the lemma in case $i=1$. Every point $z$ in $\widetilde{F}_{2}$ is a difference of two points $x$ and $y$ in $(A_{2}\cup-A_{2})^2$. Moreover 
$x$ and $y$ necessarily belong to the same connected component of $(A_{1}\cup-A_{1})^2$, because for $a$ and $b$ in different components of $(A_{1}\cup-A_{1})$, there holds $|a-b|\geq\frac{1}{2}$, while $|z|<\frac{1}{2}$. So we can assume that $x$ and $y$ belong to the connected component $D$ of $(A_{1}\cup-A_{1})^2$ which contains 0, because any other connected component is contained in a translated copy of $D$. Since the intersection
of $D$ with $(A_{2}\cup-A_{2})^2$ is an homothetic copy of  $(A_{1}\cup-A_{1})^2$ of ratio $\frac{1}{8-\varepsilon}$, then the same is true for their difference sets, and hence for $\widetilde{F}_2$ and $\widetilde{F}_1$.
\end{proof}

In particular, it follows that $\widetilde{C}=\widetilde{C}_1$. An easy computation shows that if $\varepsilon$
is sufficiently small, then a line $\ell$ with the claimed property exists (see Figure \ref{f1}). For example the line $$y=\frac{2+\varepsilon}{1-\frac{2}{8-\varepsilon}}x$$
has the claimed property, for $\varepsilon$ sufficiently small.

\section{The Second Conjecture}\label{sc}
\subsection {Strategy of the proof}
We consider the set $A$ constructed in the previous section and we set $S:=A\times A\times A$. We already proved that $\dim_{\mathcal{H}}(S)>1$. Then we prove that for a certain open set of directions in the plane $z=0$, any line with such direction intersect $A\times A$ in at most one point. This is sufficient to conclude that
every blow-up of $A\times A$ at any point does not contain points in such directions. Hence every blow-up of $A\times A$ is contained in a closed cone $\widetilde{C}_{xy}$ which does not contain points in such directions. Therefore every blow-up of
$S$ is contained in the closed cone $C_{xy}:=\widetilde{C}_{xy}\times\R$. The same argument applied to the coordinate planes $x=0$ and $y=0$, gives that every blow-up of
$S$ is contained in the intersection $C:=C_{xy}\cap C_{xz}\cap C_{yz}$, which does not contain planes, because its projections on the coordinate planes never coincide with the full coordinate plane itself.

\subsection {Property of $A\times A$}\label{prp}
Consider the set $A$ constructed in the previous section. We can assume here $\varepsilon=1$. The set $A\times A$ in the plane $z=0$ has the following property. For all $q\in\R$, every line $y=mx+q$ such that $\frac{7}{5}<m<5$ intersects the set $A\times A$ in at most one point. In fact it is easy to see that $A_1\times A_1$ has the property that such a line intersects at most one of its four connected components (see Figure \ref{f2}). Since the part of $A_2\times A_2$ which is contained in one connected component of $A_1\times A_1$ is homothetic to $A_1\times A_1$ itself, the same property holds for $A_2\times A_2$ and, by induction, for $A_i\times A_i$, for every $i$.\\

\begin{figure}[htbp]
\center
\input{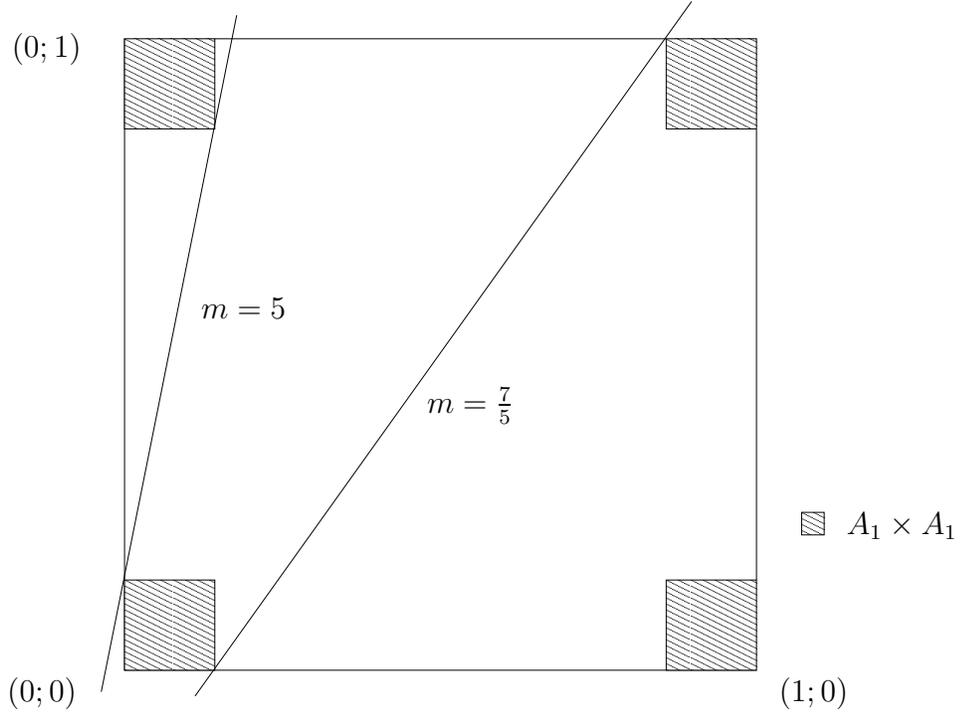}
\caption{Every line $y=mx+q$ with $\frac{7}{5}<m<5$ intersects the set $A_1\times A_1$ in at most one point. }
\label{f2}
\end{figure}

\subsection{Blow-up of a set}\label{bu}
Given a closed set $S$, a point $x\in S$ and sequence $r_j\searrow 0$, there exists a subsequence $r_{j_k}$ such that the sequence of sets $$S_{x,r_{j_k}}:=\overline{B}(0,1)\cap \frac{S-x}{r_{j_k}}$$ converges in the Hausdorff distance to some closed set $S_x$. In this case we say that $S_x$ is a blow-up of the set $S$ at the point $x$. Of course, choosing a different sequence $r_j$ or a different subsequence $r_{j_k}$ can lead to a different blow-up. 
\subsection{Blow-ups of $A\times A$}
Now we want to prove that at every point of $A\times A$, every blow-up of $A\times A$ intersects a line $y=mx$ with $\frac{7}{5}<m<5$ only at the origin. Assume the contrary that there exist $\frac{7}{5}<m_0<5$, a point $z_0=(x_0,y_0)\in A\times A$, a point $p=(t,m_0t)\neq (0,0)$ and a sequence $r_n\searrow 0$ such that the point $p$ belongs to the limit of the sequence $(A\times A)_{z_0,r_n}$. This means that every arbitrarily small neighborhood of $p$ intersects the set $(A\times A)_{z_0,r_n}$, for every $n$ sufficiently large. Take a neighborhood of $p$ which does not contain 0 and it is sufficiently small to be contained in the cone made by all lines $y=mx$ with $\frac{7}{5}<m<5$. Since this neighborhood intersects the set $(A\times A)_{z_0,r_n}$, then for some $\frac{7}{5}<m'<5$, the line $y=m'x$ contains a point of $(A\times A)_{z_0,r_n}$ which is not 0. Or, in other words, the line $y-y_0=m'(x-x_0)$ contains a point of $(A\times A)$ which is not $z_0$. Therefore that line intersects the set $(A\times A)$ in at least two point, which is a contradiction to \S \ref{prp}.\\

We can conclude that at every point of $A\times A$, every blow-up of $A\times A$ is contained in the closed cone
$$\widetilde{C}_{xy}:=\left\{(x,y)\in\R^2:y\neq mx, {\rm{for\; every\;}} \frac{7}{5}<m<5\right\}\cup 0.$$

\subsection{Blow-ups of $S$}
Given a projection $\pi$, the projection of a blow-up of a set $S$ at a point $z$ (i.e. the limit of $S_{z,r_n}$, for some sequence $r_n$) is contained in the limit of $\pi(S)_{\pi(z),r_n}$: a blow-up of the projection $\pi(S)$ at the point $\pi(z)$. In particular every blow-up of the set $S:=A\times A\times A$ at any point is contained in the cone
$$C_{xy}:=\{(x,y,z)\in\R^3:(x,y)\in\widetilde{C}_{xy}\}.$$
The same holds for the cones $C_{xz}$ and $C_{yz}$ defined analogously. Therefore every blow-up of the set $S$ at any point is contained in the cone $$C:=C_{xy}\cap C_{xz}\cap C_{yz}.$$ As in the previous section, this cone cannot contain any plane, because, if so, at least one of the projections of the cone on the coordinate planes should be full, which is not the case. We deduce that the set $S$ has the following property. Every blow-up of $S$ at every point of $S$ is contained in a closed cone (which is in particular a trivial compact family of closed cones) which does not contain planes. Nevertheless $\dim_{\mathcal{H}}(S)>1$. Therefore $S$ disproves the second conjecture.

This argument, together with Theorem \ref{t1}, gives indirectly a counterexample to the First Conjecture. In fact the cone $C$ must have building dimension strictly larger than 1, because it contains every blow-up at every point of the set $S$ (which has dimension strictly larger than 1).

%
%

\bibliographystyle{plain}

%
%

\vskip .5 cm

{\parindent = 0 pt\begin{footnotesize}

A.M.
\\
Max-Planck-Institut f\"ur Mathematik
in den Naturwissenschaften
\\
Inselstrasse~22,
04103 Leipzig,
Germany
\\
e-mail: {\tt marchese@mis.mpg.de}

\end{footnotesize}
}

\end{document}